\documentclass[11pt,a4paper]{article}
\usepackage[utf8]{inputenc}
\usepackage[T1]{fontenc}
\usepackage{newtxtext}
\usepackage{newtxmath}
\usepackage{algorithm}
\usepackage{algorithmic}
\usepackage{enumitem}
\usepackage[capitalize,nameinlink]{cleveref}

\usepackage{xcolor}

\usepackage{graphicx}
\usepackage{amsmath,amssymb,amsthm}
\newcommand{\RR}{\mathbb{R}}
\newcommand{\R}{\mathbb{R}}
\newcommand{\Q}{\mathbb{Q}}

\newcommand{\T}{^{\mathrm{T}}}
\newcommand{\ux}{\underline{x}}
\newcommand{\ox}{\overline{x}}
\newcommand{\A}{\mathcal{A}}
\newcommand{\Z}{Z}

\newcommand{\MH}[1]{{\textcolor{blue}{#1}}}

\DeclareMathOperator{\rank}{rank}

\newcommand{\linclen}{q}
\newcommand{\rankQ}{r}
\newcommand{\diag}{\text{diag}}
\newcommand{\ah}{H}

\newtheorem{observation}{Observation}
\newtheorem{theorem}{Theorem}

\begin{document}

\title{A new polynomially solvable class of quadratic optimization problems with box constraints\thanks{This work was supported by Czech Science Foundation (first author: project 18-04735S, second author: project 19-02773S, third author: project 17-13086S).}
}


\author{Milan Hlad\'{\i}k\footnote{
Department of Applied Mathematics, Faculty of Mathematics and Physics, Charles University, Malo\-stransk\'e n\'am. 25, 118 00 Prague, Czech Republic, 
\texttt{hladik@kam.mff.cuni.cz}}
 \and
        Michal \v{C}ern\'{y}\footnote{Department of Econometrics, Faculty of Informatics and Statistics, University of Economics, Prague, W. Churchill's Sq. 4, 130 67 Prague, Czech Republic, 
\texttt{cernym@vse.cz}}
\and
        Miroslav Rada\footnote{Department of Financial Accounting and Auditing, Faculty of Finance and Accounting; Department of Econometrics, Faculty of Informatics and Statistics, University of Economics, Prague, W. Churchill's Sq. 4, 130 67 Prague, Czech Republic, 
\texttt{miroslav.rada@vse.cz}}         
}



\maketitle

\begin{abstract}
We consider the quadratic optimization problem 
$\max_{x \in C}\ x\T Q x + \linclen\T x$, where $C\subseteq\RR^n$ is a box and $\rankQ \coloneqq \rank(Q)$ is assumed to be $\mathcal{O}(1)$ (i.e., fixed). We show that this case can be solved in polynomial time for an arbitrary $Q$ and $\linclen$. The idea is based on a reduction of the problem to enumeration of faces of a certain zonotope in dimension $O(\rankQ)$. This paper generalizes previous results where $Q$ had been assumed to be positive semidefinite and no linear term was allowed in the objective function. Positive definiteness was a strong restriction and it is now relaxed.  
Generally, the problem is NP-hard; this paper describes a new polynomially solvable class of instances, larger than those known previously.
\end{abstract}

\section{Introduction}
\paragraph{The problem.}
We consider the quadratic optimization problem 
\begin{equation}
    \max_{x\in \RR^n}\ x\T Q x + \linclen\T x
\quad \text{s.t.}
\quad 
\ux \leq x \leq \ox, 
\label{eq:basic}
\end{equation}
where $Q \in \RR^{n \times n}$ is an arbitrary matrix,
$\linclen, \ux, \ox \in \RR^n$ and 
$\ux \leq \ox$. 
We will also use the symbols $f(x) \coloneqq x\T Q x + \linclen\T x$ and $C \coloneqq \{x\in\RR^n\ |\ 
\ux \leq x \leq \ox\}$.

 Recall that \cref{eq:basic} is solvable in polynomial time if $-Q$ is positive semidefinite (psd); otherwise the problem is NP-hard. It is known that even
a single negative eigenvalue of $-Q$ suffices for NP-hardness 
\cite{ParVav1991,vavasis:1991:Nonlinearoptimizationcomplexity}.

\paragraph{The rank-deficient case and our contribution.}
We show that \cref{eq:basic} can be solved in polynomial time
if $Q$ is rank-deficient. In the entire text we assume that
\begin{equation}
\rankQ := \text{rank}(Q) = \mathcal{O}(1). 
\label{eq:rankdef}
\end{equation}
 Our method works for an arbitrary $\rankQ$; but its complexity can be proven to be polynomial only under the assumption \cref{eq:rankdef}. 
 
 Generally, our method is 
``the faster, the lower $\rankQ$ is''. This is also why
we will discuss in \cref{sLowRankAch} 
whether $Q$ can be replaced by another
matrix $Q'$ such that $\rank(Q') < \rank(Q)$
and 
$x\T Q x = x\T Q' x$ for all $x$. (This is a problem of finding 
a good representation of a given quadratic form. We point out that, in our context, representation of a quadratic form by a symmetric matrix---as it is usual in literature---\emph{need not} be a good idea in general.)

\paragraph{Related work.} Since the class of problems \cref{eq:basic} is NP-hard, a search for polynomially solvable sub-classes, as large as possible, is relevant. This paper generalizes the work of Allemand et al. and Ferrez et al.~\cite{allemand:2001:polynomialcaseunconstrained,ferrez:2005:Solvingfixedrank} who consider
the problem \cref{eq:basic}
with additional assumptions of $\linclen = 0$ and $Q$ psd, while our method works with an arbitrary $Q$ and $\linclen$. 

To be more precise, papers \cite{allemand:2001:polynomialcaseunconstrained,ferrez:2005:Solvingfixedrank} also consider binary variables 
$x \in \{0,1\}^n$ instead of continuous. However, since they work with the psd case only, the constraints $x \in \{0,1\}^n$ are
equivalent to $0 \leq x \leq e$ with continuous variables $x$, where $e = (1, \dots, 1)\T$. Here, it is worth noting that our method contributes to the psd binary case, too, since it admits a linear term in the objective function.

\section{The method}\label{sAlg}
Since the feasible region is compact, the problem \cref{eq:basic} attains a maximum. A maximizer is a stationary point in the following sense. Let $F$ be a face of $C$, let $x^*\in F$, and let $H(F)$ be the affine hull of $F$. Then $x^*$ is called \emph{a stationary point with respect to face $F$} if it is a stationary point with respect to the affine space $H(F)$ (i.e., the partial derivatives of the objective function at $x^*$ vanish in the direction of $H(F)$).
We say that a feasible point $x^*$ is \emph{a stationary point} if it is a stationary point with respect to a certain face~$F$.

We recall basic properties of stationary points.
By definition, each vertex of $C$ is a stationary point.
If $F$ and $F'$ are faces of $C$ and $x^*\in F\subseteq F'$ is a stationary point with respect to $F'$, then $x^*$ is a stationary point with respect to~$F$. Therefore, we can restrict to the minimal (w.r.t.\ inclusion) face containing~$x^*$; point $x^*$ then lies in the relative interior of this face.
Notice that the converse direction does not hold in general: If $x^*\in F\subseteq F'$ is a stationary point with respect to $F$, then $x^*$ need not be a stationary point with respect to~$F'$. 
From the KKT necessary condition, we have:

\begin{observation}\label{obsOptStat}
If $x^*\in C$ is a local optimum of \cref{eq:basic}, then it is a stationary point.
\end{observation}

The objective function $f(x)=x\T Q x + \linclen\T x$ is quadratic. Hence, the set of stationary points of $f(x)$ in every affine subspace is again an affine subspace. In particular, the stationary points of $f(x)$ with respect to a face $F$ of $C$ form a polyhedron, say~$P$. 
Since the partial derivatives are vanishing on the subspace generated by~$F$, we immediately have:

\begin{observation}\label{obsStatConst}
For each $x,y\in P$ we have $f(x)=f(y)$.
\end{observation}

This allows us to use the following basic method: Examine faces one by one, for every face, check whether there is a stationary point with respect to that face, and evaluate the objective function in such a (representative) point. There are $3^n$ nonempty faces, which makes such a method exponential in $n$. 

Note that this approach is conceptually analogous to KKT conditions. Since the objective function is quadratic and all constraints in \cref{eq:basic} are affine functions, the KKT system can be decomposed by complementarity constraints to some linear subsystems. In every such subsystem, complementarity constraints cause that some variables are fixed at their lower bounds,
some variables are fixed at their upper bounds and some remain free inside their bounds. Also, every such subsystem corresponds to one face of $C$.

To overcome the exponentiality in $n$, we transform problem \cref{eq:basic} to a form where the feasible region is a special polytope---a zonotope---in dimension $2r$ (recall that $r$ is rank of $Q$).
The transformation reduces the number of faces of the feasible region to $O(n^{2r-1})$. With the ``simpler'' feasible region, we perform the above sketched steps. Our method is described in the rest of this section. In \crefrange{sec:hiding}{sec:handling}, we address the following steps in more detail:
\begin{enumerate}
    \item hiding the linear term $\linclen\T x$ from the objective function,
    \item reducing the dimension of the problem,
    \item enumerating faces of the feasible region in the reduced dimension, and
    \item finding stationary points with respect to each face and examining them.
\end{enumerate}
In \cref{sec:algorithm}, the algorithm is summarized and its correctness is proved. 

\subsection{Hiding the linear term\label{sec:hiding}}
The dimension-reduction step described in \cref{sec:reducing} works only if the objective function $f(x)$ contains no linear term. 
We show that an instance of the problem \cref{eq:basic}, can be transformed to the case with $\linclen=0$, i.e.
\begin{equation}
\max\ x\T Q x \ \ \text{s.t.}\ \ x \in C = \{\xi \ |\ \ux \leq \xi \leq \ox\}.
\label{eq:myproblemtwo}
\end{equation}
The transformation is performed at a moderate cost: the number of variables increases by $1$ and the rank of the quadratic form is increased by at most $2$.

We introduce an auxiliary variable $w \in \RR$. We get a new 
$(n+1)$-vector of variables
$x' = \binom{x}{w}$. We set 
\begin{equation}
    \label{eq:hiding:symbols}
Q' := \left(
\begin{matrix}
Q & \frac{1}{2}\linclen \\
\frac{1}{2}\linclen\T & 0
\end{matrix}
\right), \quad
\ux' = \binom{\ux}{1},
\quad
\ox' = \binom{\ox}{1}.
\end{equation}
Now \cref{eq:myproblemtwo} can be rewritten as
\begin{equation}
\max_{x'\in\RR^{n+1}}\ (x')\T Q' x' \ \ \text{s.t.}\ \ \ux' \leq x' \leq \ox',
\label{eq:modified}
\end{equation}
since the constraints imply $w=1$ and therefore 
$(x')\T Q' x' = x\T Q x + w \linclen\T x  = x\T Q x + \linclen\T x$.

Observe that
\begin{equation}
\rank(Q') \leq \rank(Q) + 2,
\label{eq:rankQ}
\end{equation}
so the transformation does not violate assumption 
\cref{eq:rankdef}.
From now on we will, without loss of generality, assume $\linclen = 0$.

\subsection{Reducing the dimension\label{sec:reducing}}

We use an approach conceptually similar to \cite{allemand:2001:polynomialcaseunconstrained}.
We perform \MH{a} rank factorization of $Q$ and construct a pair of matrices
$U, V \in \RR^{\rankQ \times n}$ such that $Q = U\T V$. An adaptation of Gaussian elimination can be used for this purpose; see e.g.\ \cite{stewart:1998:MatrixAlgorithmsVolume}.
Then we
introduce matrices $G = \binom{U}{V}$ (i.e.,~$G$ results from sticking $V$ under $U$) and 
\begin{equation}
W = \frac{1}{2}\left(
\begin{matrix}
0_{\rankQ \times \rankQ} && I_{\rankQ \times \rankQ} \\
I_{\rankQ \times \rankQ} && 0_{\rankQ \times \rankQ}
\end{matrix}
\right).
\label{eq:W}
\end{equation}
Using substitutions $u = U\T x$,
$v = V\T x$ and stacking the new variables $u,v$ into a
vector $y = \binom{u}{v}$, we get
\begin{align}
\max_{x\in C}\ x\T Q x 
&= 
\max_{x\in C}\ x\T U\T V x
\nonumber
\\ & =
\max_{\substack{x\in \RR^n \\ u,v\in\RR^{\rankQ}}}
\{u\T v\ |\ u = U x,\ v = V x,\ x \in C\}
\nonumber
\\ & =
\max_{\substack{x\in \RR^n \\ u,v\in\RR^{\rankQ}}}
\left\{u\T v\ \Big|\ 
\binom{u}{v} = \binom{U}{V}x,\ x \in C\right\}
\nonumber
\\ & =
\max_{\substack{x\in\RR^n \\ y \in \RR^{2\rankQ}}}\{y\T W y\ |\ 
y = Gx,\ x \in C\}
\label{eq:projection:hypercube}
\\ &
 =
\max_{y \in \RR^{2\rankQ}}\{y\T W y\ |\ 
y \in \Z(G,C)\},
\label{eq:reduced}
\end{align}
where
\begin{equation}\Z(G,C)\coloneqq \{y \in \R^{2r}\ |\ y=Gx,\ x\in C\}. \label{eq:zono}\end{equation} 
Notice that $\Z(G,C)$ is the image of the box $C$ under the linear mapping $x \mapsto Gx$. 
It represents a special polytope, the so called \emph{zonotope}; see e.g.\ \cite[chapter 7]{ziegler:2012:LecturesPolytopes}.

Note that the feasible region in \cref{eq:projection:hypercube} is of dimension $n+2r$. The form \cref{eq:reduced} shows that the dimension is reduced to ``only'' $2r$.
This is for the price that the feasible region is no longer explicitly represented as an intersection of halfspaces, but its representation is in a sense implicit now: it is given by the bounds $\ux, \ox$ of the box $C$ and by the matrix $G$. Hence, classical algorithms that deal with halfspace representation cannot be used directly without falling back to the dimension $O(n)$.

\subsection{Enumerating faces\label{sec:enumerating}}
Here, we use the incremental algorithm by Edelsbrunner et al. (see \cite{edelsbrunner:1986:Constructingarrangementslines} or \cite[chapter 7]{edelsbrunner:2005:Algorithmscombinatorialgeometry}), which enumerates all faces (regions) of an arrangement of $m$ hyperplanes in $\R^d$ in time $O(m^d)$, provided that $d = O(1)$. Since there are arrangements with $O(m^d)$ faces (see \cite{buck:1943:Partitionspace} or \cite{zaslavsky:1975:FacingarrangementsFacecount}), the algorithm is asymptotically optimal.

The algorithm can be adapted for enumerating faces of a zonotope. There is the well-known duality between zonotopes and hyperplane arrangements: faces of a zonotope $\Z(G,C)$ are in one-to-one correspondence with regions of the central arrangement $\A(G)$ of hyperplanes $g_i\T y = 0$ for $i = 1, \ldots, n\}$, where $g_i$ is the $i$th column of $G$. 
Note that $\A(G)$ is in dimension $2r$.
More precisely, the face lattice of $\Z(G,C)$ is isomorphic to the lattice of regions of $\A(G)$.
This duality between zonotopes and central hyperplane arrangements has a nice geometric background: actually, $\A(G)$ forms a normal fan of $\Z(G,C)$. For a $k$-dimensional face $F$ of $\Z(G,C)$, there is a $(2r-k)$-dimensional region in $\A(G)$ that forms the normal cone of~$F$. This topic is discussed in more detail e.g.\ in \cite[chapter 7]{ziegler:2012:LecturesPolytopes} or \cite[chapter 1]{edelsbrunner:2005:Algorithmscombinatorialgeometry}.

\paragraph{Sign vectors, representation of regions of arrangements and faces of zonotopes}
The algorithm by Edelsbrunner et al.~\cite{edelsbrunner:1986:Constructingarrangementslines}
characterizes regions by \emph{sign vectors}. For a point $y \in \R^{2r}$, its sign vector $\sigma(y)\in\{-1,0,1\}^n$ says where $y$ is located with respect to each of $n$ hyperplanes of $\A(G)$. More precisely:
\begin{equation}
    \label{eq:sign:vector}
    \sigma_i(y) \coloneqq \left\{
        \begin{array}{l}
            \phantom{-}1 \quad \text{if } g_i\T y > 0,\\
            \phantom{-}0 \quad \text{if } g_i\T y = 0,\\
            -1 \quad \text{otherwise}.
        \end{array}
    \right.
\end{equation}
A region of $\A(G)$ is simply the set of points with same sign vector.

By duality of zonotopes and arrangements, a sign vector characterizing  a region of $\A(G)$ can also be used to characterize the corresponding face of $\Z(G,C)$. However, we will need a more explicit representation of faces of $\Z(G,C)$ in the next section. Since every face of the box $C$ is also a box, every face of $\Z(G,C)$ is again a zonotope. For a sign vector $\sigma$, define the box 
\begin{equation}
C_\sigma \coloneqq C \cup \ah(\sigma),\label{eq:C:sigma}
\end{equation}
where  
\begin{equation}    
    \ah(\sigma) \coloneqq
    \left\{ x \in \R^n\ \left|\ \begin{array}{l}
        x_i  = \ox_i \text{ if } \sigma_i =1, \\
        x_i  = \ux_i \text{ if }\sigma_i = -1,
    \end{array}
      \right.\ i = 1,\ldots, n\right\}
        \label{eq:S:sigma}\end{equation}
        is the affine subspace fixing some coordinates either to $\ux_i$ or $\ox_i$, according to the sign vector $\sigma$. 
Note that $\ah(\sigma)$ is actually the affine hull of $C_\sigma$.

Now, the zonotope $\Z(G, C_\sigma)$ is the face of $\Z(G,C)$ characterized by $\sigma$.
By the zonotope--arrangement duality, the face $\Z(G, C_\sigma)$ corresponds to the region of $\A(G)$ characterized by $\sigma$.

\paragraph{Taking advantage of central symmetry}
Actually, since $\A(G)$ is a central arrangement, its combinatorial complexity (measured by the number of regions) is of the same order as of an arrangement of $(n-1)$ hyperplanes in $\R^{2r-1}$. This can be verified by a simple argument: regions of $\A(G)$ can be classified in three groups with respect to the last sign of their sign vectors. To enumerate regions of each of these groups, it is sufficient to consider the intersection of $\A(G)$ with hyperplane $g_n\T y = -1$, then with hyperplane $g_n\T y=0$ and finally with $g_n\T y=1$. Each of these intersections forms an arrangement of $(n-1)$ hyperplanes in $\R^{2r-1}$. The crucial property is that each of these hyperplanes intersects \emph{all} regions of $\A(G)$ with the respective last sign.

Note also that if $\A(G)$ contains a region with sign vector $\sigma$, then there is also the region with sign vector $-\sigma$ by central symmetry, hence the regions with $\sigma_n=-1$ can be obtained from the regions with $\sigma_n = 1$.

Working with arrangements in $\R^{2r-1}$ allows us to claim to achieve the time complexity of the enumeration to be $O(n^{2r-1})$ for $r \geq 2$. 
If $r=1$, zonotope $\Z(G,C)$ is $2$-dimensional and its faces can be enumerated trivially.

\subsection{Handling one face\label{sec:handling}}
Say that a face $F = \Z(G,C_\sigma)$ is of interest. 
First, we find the affine hull $\ah(F)$ of~$F$, and then we describe the set of stationary points with respect to $F$. If this set is not empty, we compute a representative point of it.

We define of some auxiliary symbols first. Without loss of generality assume that $\sigma$ is such that its first $k$ signs are zero, the others are nonzero. Split the matrix $G \in \R^{2r\times n} = (G_A, G_b)$ to matrices $G_A \in \R^{2r\times k}$ and $G_b \in \R^{2r \times (n-k)}$, and similarly split the vectors $\ux = \binom{\ux_A}{\ux_b}$, $\ox = \binom{\ox_A}{\ox_b}$, $x = \binom{x_A}{x_b}$ and $\sigma =\binom{\sigma_A}{\sigma_b}$ to pairs of vectors with $k$ and $n-k$ elements. Say that $A \in \R^{\ell\times 2r}$ is a matrix such that its rows form a basis of the orthogonal complement of the column space of $G_A$.
Hence, we have $A G_A =0$, $\ell = 2r - \rank(G_A)$ and $\rank(A) = \ell$. 

Now, the preimage $C_\sigma$ of face $F = \Z(G,C_\sigma)$ of $Z(G,C)$ belongs to the affine subspace $\ah(\sigma)$. Hence $F \subseteq \ah(F) = \{y\ |\ y = Gx,\ x \in \ah(\sigma)\}$. Using the new notation, we can write 
\begin{align}
    \ah(F) &= \left\{y\ |\ y = G_A x_A + G_b x_b,\ \binom{x_A}{x_b} \in \ah(\sigma) \right\}\\
         &= \left\{y\ |\ Ay = 0 x_A + A G_b x_b,\ x_b= \textstyle\frac{1}{2}\displaystyle(\ux_b+\ox_b)+\textstyle\frac{1}{2}\displaystyle\diag(\sigma_b)(\ox_b-\ux_b)\right\}\\
        &= \left\{y\ |\ Ay = Ab\right\},
\end{align}
where $b \coloneqq \frac{1}{2} G_b \left((\ux_b+\ox_b)+\diag(\sigma_b)(\ox_b-\ux_b)\right)$.

By the Lagrange multiplier method applied to (\ref{eq:reduced}), we get that if $y \in F$ is a maximizer of \eqref{eq:reduced}, it has to satisfy the linear system
\begin{equation}
    A\T \lambda = 2W y,\ Ay = Ab,\ y \in F = Z(G, C_\sigma).\label{eq:stationary:initial}
\end{equation}
Note that since $F \subseteq \ah(F)$, constraints $Ay = Ab$ are redundant. After elimination of the $y$-variables we obtain system \cref{eq:stationary:point}. 
\begin{gather}
    A\T \lambda - 2W G_A x_A = 2W b,\ \ux_A \le x_A \le \ox_A.\label{eq:stationary:point}
\end{gather}

The system \cref{eq:stationary:point} can be solved as a linear program (LP) with variables $x_A \in \R^k,\ \lambda \in \R^\ell$.
From its solution $(\lambda^*,x^*_A)$, we can build the vector $x^* = \binom{x^*_A}{x_b}$, which is actually a stationary point of $f(x)$ with respect to~$C(\sigma)$. This is because $x^*$ is a preimage of the stationary point $y^* = Gx^*$ of $y\T W y$; if $y^*$ is a maximizer of \cref{eq:reduced}, then $x^*$ is a maximizer of \cref{eq:basic} and vice versa.

\subsection{The algorithm\label{sec:algorithm}}
The steps from the previous sections are summarized in \cref{alg:algorithm}. \cref{the:complexity,the:correctness} are about complexity and correctness of the algorithm.
The input data $(Q, \underline{x}, \overline{x})$ are assumed to be rational.

\begin{algorithm}[b]
    \begin{algorithmic}[1]
        \normalsize
        \REQUIRE $Q \in \Q^{n\times n},\ \linclen,\ux,\ox \in \Q^{n}$, $r \coloneqq \rank(Q)$
        \STATE $f^* \coloneqq -\infty,\ x^* \coloneqq 0$\label{alg:setting:prelim}
        \STATE \textbf{if} $\linclen \not = 0$ \textbf{then}\label{alg:hiding:condition}
        \STATE \qquad
        $Q \coloneqq \begin{pmatrix}Q&\frac{1}{2}\linclen\\\frac{1}{2}\linclen\T&0\end{pmatrix};\ \ux \coloneqq \binom{\ux}{1};\ \ox \coloneqq \binom{\ox}{1}$\hfill \COMMENT{{\small elimination of the linear term (Sec. \ref{sec:hiding})}}\label{alg:hiding}
        \STATE compute $U,V \in \Q^{r\times n}$ such that $U\T V = Q$;\ $G \coloneqq \binom{U}{V}$\hfill\COMMENT{{\small reduction of dimension (Sec.~\ref{sec:reducing})}}\label{alg:reducing}
        \STATE compute the set $S$ of sign vectors of regions of $\A(G)$ \hfill\COMMENT{{\small enumeration of faces (Sec. \ref{sec:enumerating}})}\label{alg:enumerating}
        \STATE\textbf{for each }$\sigma \in S$ \textbf{do}\hfill\COMMENT{{\small handling one face (Sec. \ref{sec:handling})}}\label{alg:handling:for:start}
        \STATE \qquad construct the affine hull $Ay = Ab$ of $F(\sigma)$\label{alg:handling:basis}
        \STATE \qquad solve LP \cref{eq:stationary:point} to obtain preimage $x$ of a stationary point w.r.t.\ face $F(\sigma)$\label{alg:handling:lp}
        \STATE \qquad\textbf{if} $x\T Q x > f^*$ \textbf{then}\label{alg:handling:update:condition}
        \STATE \qquad \qquad $x^* \coloneqq x$; $f^* \coloneqq x\T Q x$\label{alg:handling:update}
        \STATE\textbf{end for}\label{alg:handling:end}
        \RETURN $f^*,\ x^*$\label{alg:return}
    \end{algorithmic}
    \caption{The algorithm}
    \label{alg:algorithm}
\end{algorithm}

\begin{theorem} 
    \label{the:complexity}
If $\linclen=0$, \Cref{alg:algorithm} works in time $O(n^{2r-1}\cdot \text{lp}(Q,\ux,\ox))$, where $\text{lp}(Q,\ux,\ox)$ is time needed to solve a linear program with $O(n)$ variables, $O(n)$ constraints and data with bit-size polynomially bounded by the bit-sizes of $Q,\ux,\ox$. 
If $\linclen\not=0$, \Cref{alg:algorithm} works in time $O(n^{2r+1}\cdot \text{lp}(Q,q,\ux,\ox))$,
\end{theorem}

\begin{proof}
Let $\linclen=0$. In step \ref{alg:enumerating},
 $O(n^{2r-1})$ faces of the zonotope are enumerated. This takes time of same order, i.e.\ $O(n^{2r-1})$. For every face, linear program of the form \cref{eq:stationary:point} is solved. The program has at most $n+2r$ variables and at most $2n+2r$ constraints. The bit-size of the data of the LP depends on computations performed in step~\ref{alg:hiding} (negligible), step~\ref{alg:reducing} (rank factorization, polynomially solvable via adaptation of Gaussian elimination) and step~\ref{alg:handling:basis} (finding orthogonal subspace, polynomially solvable). Since all steps are performed in polynomial time, the bit-size of data of the LP \cref{eq:stationary:point} is bounded by a polymonial in the bit-sizes of $Q,\ux,\ox$.

If $\linclen\not=0$, then the time complexity follows from the construction in step~\ref{alg:hiding} and from the previous case.
\end{proof}

\begin{theorem}\label{the:correctness}
The method is correct and solves \cref{eq:basic}.
\end{theorem}

\begin{proof}
    Since $Z(G,C)$ is a compact set and $y\T W y$ defined in \cref{eq:reduced} is a continuous function, a maximizer always exists. 

    Now consider a~maximizer $y^*$. Let $\sigma$ be a sign vector of the minimal face 
$F$ of $Z(G,C)$ containing $y^*$. 
Let $Ay = Ab$ be the affine hull of $F$ as constructed in step~\ref{alg:handling:basis}. Since $y^*$ is a maximizer, it satisfies the first-order Lagrange conditions $A\T\lambda = 2Wy$, $Ay = Ab$.
And, since $y^* \in F$, it also satisfies the last condition in 
\cref{eq:stationary:initial}. System \cref{eq:stationary:initial} is tested (in form \cref{eq:stationary:point}) in step~\ref{alg:handling:lp} in the iteration when $\sigma$ is chosen from $S$.

What can happen is that the linear programming solver for \cref{eq:stationary:point} finds
a point $x^{**}$ such that $Gx^{**} \eqqcolon y^{**} \neq y^*$. But this does not matter since 
\begin{equation}
    (y^*)\T W y^* = (y^{**})\T W y^{**}
\label{eq:rovnost}
\end{equation}
in view of Observation~\ref{obsStatConst}.
\end{proof}

\section{Minimal rank achievement}\label{sLowRankAch}
As we saw in previous sections, the smaller rank of $Q$ the better. Here we focus on obtaining a matrix with minimal rank that carries the same information as $Q$. We start with some illustrative examples and preliminary observations.

Requiring $Q$ to be symmetric need not be the best choice. The symmetrization $Q\mapsto \frac{1}{2}(Q+Q\T)$ may both increase or decrease the rank of the resulting matrix. Obviously, the rank cannot increase more than twice, so we have the upper bound
$$
\rank(Q+Q\T)\leq 2\rank(Q).
$$
The bound is tight, e.g., for the block diagonal matrix with uniform blocks in the form
$$
\begin{pmatrix}
0 & 1\\0 & 0
\end{pmatrix}.
$$
On the other hand, rank can decrease arbitrarily. Consider the block diagonal matrix $A$ with blocks in the form
$$
\begin{pmatrix}
0 & 1\\-1 & 0
\end{pmatrix}.
$$
Then $Q$ is nonsingular, but $Q+Q\T$ has zero rank. Even upper triangular matrices have similar behaviour. The symmetric counterpart of the nonsingular matrix
$$
Q=\begin{pmatrix}
1 & 2 & \dots & 2\\
0 & 1 & \ddots & \vdots\\
\vdots & \ddots & \ddots & 2\\
0 & \dots & 0 & 1
\end{pmatrix}.
$$
is the all-ones matrix of rank one.

Now, we show what is the smallest achievable rank. The proof is constructive, so it can be used for a practical construction of the corresponding matrix.

\begin{theorem}
Let $(p,q,s)$ be the signature of the quadratic form $x\T Qx$. Then the smallest achievable rank of its matrix is $\max(p,q)$.
\end{theorem}

\begin{proof}
Without loss of generality assume that $p\geq q$ and $Q$ is symmetric. 
First, we show that the rank of $p$ is achievable. 
In a polar basis, the matrix of the quadratic form is diagonal with diagonal entries
$$
1,-1,\dots,1,-1,1,\dots,1,0,\dots,0.
$$
We replace each diagonal block
$$
\begin{pmatrix}
1 & 0\\0 & -1
\end{pmatrix}
$$
of rank 2 by the block
$$
\begin{pmatrix}
1 & 1\\-1 & -1
\end{pmatrix}
$$
of rank 1. Thus, we obtain a matrix of the overall rank~$p$. The matrix corresponds to the polar basis, but by the transformation to the canonical basis we get the required matrix.

Now, we show that there is no matrix of the quadratic form with smaller rank. Suppose to the contrary that there is some of rank $r<p$. Then the quadratic form vanishes on a subspace of dimension $n-r$. However, the quadratic form is positive definite on a subspace of dimension $p$. Therefore, there is a nontrivial subspace, where the form is both zero and positive definite; a contradiction.
\end{proof}

\section{Concluding remarks}

The proposed method can solve arbitrary rank-deficient quadratic maximization in polynomial time, extending previous result by \cite{allemand:2001:polynomialcaseunconstrained}, which worked only in the psd case with no linear term. As a direct consequence, the method can also be used in binary psd case.


The enumeration method from Section 2.3 requires the entire face lattice to be stored in memory. Thus, the method can be considered as ``memory intensive'', although still polynomial as far as the rank is low.
Recall that the enumeration method from \cite{allemand:2001:polynomialcaseunconstrained} prints the enumeration as a stream without the necessity to remember the history; this property reduces the space complexity significantly. Unfortunately, we currently cannot achieve an analogous property for~\cref{eq:basic}.
Since all faces of the feasible region have to be stored in memory during the enumeration, the algorithm is very memory intensive (although still polynomial as long as the
rank of the quadratic form is fixed).
Finding a memory less expensive algorithm thus remains as a challenging open problem.

Since the complexity of our method depends on the rank of the underlying quadratic form, it might be useful to represent it by a nonsymmetric matrix. In \cref{sLowRankAch}, we showed a method how to achieve the smallest possible rank. 

\bibliographystyle{spmpsci}      

\bibliography{quadratic_fixed_poly}


\end{document}